\documentclass{amsart}
\usepackage{amsmath, amssymb, amsthm}
\usepackage{amscd}
\usepackage{graphicx} 
\usepackage{enumerate}

\title[On equidistribution and heights]{Equidistribution and the heights of totally real and totally $p$-adic numbers}
  \author[Fili]{Paul Fili}
 \address{Department of Mathematics\\ University of Rochester, Rochester, NY 14627}
 \email{fili@math.rochester.edu}
 \author[Miner]{Zachary Miner}
\address{Department of Mathematics\\ University of Texas at Austin, TX 78712}
\email{zminer@math.utexas.edu}
\date{\today}

\usepackage[pdftitle={Equidistribution and the heights of totally real and totally p-adic numbers},pdfauthor={Fili and Miner},pdfstartview={}]{hyperref}

\usepackage[all]{xy}

\newtheorem{thm}{Theorem}

\newtheorem{cor}[thm]{Corollary}

\newtheorem*{thm*}{Theorem}
\newtheorem*{alg*}{Algorithm}
\newtheorem*{lemma*}{Lemma}

\theoremstyle{remark}

\newtheorem*{rmk*}{Remark}

\newtheorem*{notation*}{Notation}

\theoremstyle{definition}
\newtheorem{defn}{Definition}
\newtheorem*{defn*}{Definition}

\newcommand{\mybf}{\mathbb}

\newcommand{\bJ}{\mybf{J}}
\newcommand{\bU}{\mybf{U}}
\newcommand{\bP}{\mybf{P}}
\newcommand{\bR}{\mybf{R}}

\newcommand{\bC}{\mybf{C}}

\newcommand{\bQ}{\mybf{Q}}

\newcommand{\bA}{\mybf{A}}

\newcommand{\cD}{\mathcal{D}}

\newcommand{\al}{\alpha}

\providecommand{\abs}[1]{\lvert#1\rvert}

\newcommand{\Gal}{\operatorname{Gal}}

\newcommand{\ra}{\rightarrow}

\newcommand{\ep}{\epsilon}

\newcommand{\Tor}{\operatorname{Tor}}

\newcommand{\Qbar}{\overline{\mybf{Q}}}
\newcommand{\Kbar}{\overline{K}}

\newcommand{\Qab}{\mybf{Q}^\text{ab}}

\def\talltareesidedbox#1{\setbox0=\hbox{$#1$}\dimen0=\wd0 \advance\dimen0 by3pt\rlap{\hbox{\vrule height10pt width.4pt
 depth2pt \kern-.4pt\vrule height10.4pt width\dimen0 depth-10pt\kern-.4pt \vrule height10pt width.4pt depth2pt}}
 \relax \hbox to\dimen0{\hss$#1$\hss}}
\def\tareesidedbox#1{\setbox0=\hbox{$#1$}\dimen0=\wd0 \advance\dimen0 by3pt\rlap{\hbox{\vrule height8pt width.4pt
 depth2pt \kern-.4pt\vrule height8.4pt width\dimen0 depth-8pt\kern-.4pt \vrule height8pt width.4pt depth2pt}}
\relax \hbox to\dimen0{\hss$#1$\hss}}
\def\shorttareesidedbox#1{\setbox0=\hbox{$#1$}\dimen0=\wd0 \advance\dimen0 by3pt\rlap{\hbox{\vrule height7pt width.4pt
 depth2pt \kern-.4pt\vrule height7.4pt width\dimen0 depth-7pt\kern-.4pt \vrule height7pt width.4pt depth2pt}}
 \relax \hbox to\dimen0{\hss$#1$\hss}}

\newcommand{\hhat}{\hat{h}}

\newcommand{\Res}{\operatorname{Res}}

\newcommand{\Berk}{\mathrm{Berk}}

\newcommand{\Etor}{E_{\text{tors}}}

\newcommand{\PrePer}{\operatorname{PrePer}}

\newcommand{\sP}{\mathsf{P}}
\newcommand{\sA}{\mathsf{A}}

\begin{document}

\begin{abstract}
 C.J. Smyth was among the first to study the spectrum of the Weil height in the field of all totally real numbers, establishing both lower and upper bounds for the limit infimum of the height of all totally real integers and determining isolated values of the height. Later, Bombieri and Zannier established similar results for totally $p$-adic numbers and, inspired by work of Ullmo and Zhang, termed this the Bogomolov property. In this paper, we use results on equidistribution of points of low height to generalize both Bogomolov-type results to a wide variety of heights arising in arithmetic dynamics.
\end{abstract}

\thanks{The authors would like to thank Robert Rumely and the University of Georgia for organizing the 2011 VIGRE SSP in arithmetic dynamics and for their hospitality and financial support, during which time much of this research was conducted. We would also like to thank Tom Tucker and Lukas Pottmeyer for helpful conversations.}

\maketitle

\section{Introduction and Results}
Recall that an algebraic number is said to be \emph{totally real} if all of its Galois conjugates lie in the field $\bR$ under any choice of embedding $\Qbar\hookrightarrow\bC$. For example, if $\zeta$ is a root of unity, then $\al = \zeta+\zeta^{-1}$ is a totally real number. Schinzel \cite{SchinzelTotReal} proved the following:\footnote{In fact Schinzel's result  originally implied this result for integers, however, it is easy to see that it generalizes to totally real nonintegers as well; cf. e.g., \cite{GarzaRealConjugates, HohnOnGarza}.}
\begin{thm*}[Schinzel 1973]
Let $\al\in\bQ^{\mathrm{tr}}$, $\al\neq 0,\pm 1$ be a totally real number. Then 
\[
 h(\al) \geq h\left(\frac{1+\sqrt{5}}{2}\right) = \frac{1}{2} \log \frac{1+\sqrt{5}}{2}=0.2406059\ldots
\]
where $h$ denotes the absolute logarithmic Weil height.
\end{thm*}

A natural generalization of the concept of a totally real number is that of a \emph{totally $p$-adic number}, that is, an algebraic number $\al$ all of whose Galois conjugates lie in $\bQ_p$ for any embedding $\Qbar\hookrightarrow \bC_p$, where $\bC_p$ denotes the completion of the algebraic closure of $\bQ_p$. Equivalently, we can ask that the minimal polynomial of $\al$ split over $\bQ_p$. Note that unlike $\bC/\bR$, the extension $\bC_p/\bQ_p$ is of infinite degree. Bombieri and Zannier \cite{BombieriZannierNote} proved the following analogue of the results of Smyth and Flammang:
\begin{thm*}[Bombieri and Zannier 2001]
Let $L/\bQ$ be a normal extension (possibly of infinite degree) and $S$ is the set of finite rational primes such that $L_p/\bQ_p$ is Galois (in particular, of finite degree), then
\[
 \liminf_{\al\in L} h(\al) \geq \frac{1}{2} \sum_{p\in S} \frac{\log p}{e_p(p^{f_p} + 1)}
\]
where $e_p$ and $f_p$ denote the ramification and inertial degrees of $L_p/\bQ_p$, respectively.
\end{thm*}
\noindent Notice that this immediately implies that for any such field, the set of points of zero height (i.e., for the standard Weil height, $0$ together with the set of roots of unity in the field) is in fact finite---a fact that was obvious for $\bQ^{\mathrm{tr}}$ but is far less so for a field $L$ as above.

Bombieri and Zannier, in analogy to work of Ullmo and Zhang on the Bogomolov conjecture (see \cite{ZhangPositive,Zhang,Ullmo98}), termed this property that there is a lower bound away from zero for the height of all numbers not roots of unity in the field the \emph{Bogomolov property}. Let us make the following more precise definitions. Let $K$ be a number field and $\varphi\in K(z)$ a rational map of degree at least $2$, and denote by $h_\varphi$ its associated \emph{canonical height} following \cite{CallSilverman}. We can then generalize the notion of Bombieri and Zannier and ask which extensions $F/K$ satisfy an analogue of the Bogomolov property with respect to the height $h_\varphi$. Specifically:
\begin{defn}
 We say that a field $F/K$ satisfies the \emph{strong Bogomolov property with respect to $h_\varphi$} if
 \[
  \liminf_{\al\in\bP^1(\Kbar)} h_\varphi(\al) > 0.
 \]
We say $F/K$ satisfies the \emph{weak Bogomolov property with respect to $h_\varphi$} if
 \[
  \liminf_{\substack{\al\in\bP^1(\Kbar)\\ h_\varphi(\al)>0}} h_\varphi(\al) > 0.
 \]
\end{defn}
We note that there well-known examples of fields which satisfy the weak Bogomolov property, but not the strong, specifically, Amoroso and Dvornicich \cite{AmorosoDvornicich} showed that the maximal abelian extension $\Qab$ satisfies the weak Bogomolov property, while clearly since it contains all of the roots of unity, it does not satisfy the strong Bogomolov property.

The strong Bogomolov property, unlike the weak, seems to be intimately connected to equidistribution of small points. For example, Zhang \cite{Zhang}, as a direct consequence of his theorem on the equidistribution of strict small sequences of points on an abelian variety defined over a number field at the archimedean place, was able to show the following:
\begin{thm*}[Zhang 1998]
 Let $A$ be an abelian variety defined over a number field $K$, and $\hhat$ be a N\'eron-Tate height associated to a symmetric ample line bundle on $A$. Then the set of points $A(K\bQ^{\mathrm{tr}})$ satisfies the following strong Bogomolov property:
\[
 \liminf_{x \in A(K\bQ^{\mathrm{tr}})} \hhat(x) > 0.
\]
\end{thm*}
\noindent Note that Zhang's result originally stated that $\Tor(A(K\bQ^{\mathrm{tr}}))$ is finite and there exists an $\ep>0$ such that the set
$
 \{ x\in A(K\bQ^{\mathrm{tr}}) : \hhat(x) < \ep \}
$
is finite, but this is equivalent to the positivity of the limit infimum for our choice of height.

We will establish an analogous result for the canonical heights associated to iteration of rational maps on $\bP^1$. Since we wish to allow arbitrary number fields as our base fields, we will make the following generalization of the notions of totally real and totally $p$-adic:

\begin{defn}
Fix a base number field $K$, and let $S$ be a set of places of $K$. For each $v\in S$, we choose a Galois extension $L_v/K_v$. We say that $\al$ is \emph{totally $L_S/K$} if, for each $v\in S$, all of the $K$-Galois conjugates of $\al$ lie in $L_v$ for each embedding $\Kbar\hookrightarrow \bC_v$.
\end{defn}

Notice that the set of all totally $L_S/K$ numbers forms a field $L$ which is a normal extension of $K$.

Our main result in this paper is a broad generalization of these results to the dynamical heights arising from the iteration of rational functions. We give a simple measure-theoretic criterion which characterizes the situations above.

\begin{thm}\label{thm:main}
 Let $\varphi$ be a rational map defined over a number field $K$, and fix a finite set $S$ of places of $K$ and corresponding Galois extensions $L_v/K_v$ for each $v\in S$. Let $L$ be the field of numbers all of whose $K$-Galois conjugates lie in $L_v$ for each $v\in S$, $\mu_{\varphi,v}$ be the $v$-adic canonical measure associated to $\varphi$, and $\overline{L_v}$ the topological closure of $L_v$ in the Berkovich projective line $\sP^1(\bC_v)$.
 
 If there exists a place $v\in S$ for which
\begin{equation*}\label{eqn:condition}
 \mu_{\varphi,v}(\overline{L_v}) < 1,
\end{equation*}
then $L$ satisfies the strong Bogomolov property with respect to $h_\varphi$, that is,
\[
 \liminf_{\al\in L} h_\varphi(\al)>0.
\]
 Further, if $\varphi$ is a polynomial, then the converse is also true, namely, if $L$ satisfies the strong Bogomolov property with respect to $h_\varphi$, then there exists a place $v$ of $L$ such that $\mu_{\varphi,v}(\overline{L_v}) < 1$.
\end{thm}
\noindent In specific case of the maximal totally real field of $\bQ$, an independent proof which does not require the assumption that the $\varphi$ is a polynomial in the converse is due to Pottmeyer \cite{PottThesis}.

Before we prove our theorem, let us state a few of the interesting corollaries of our result:
\begin{cor}\label{cor:preperiodic-pts}
 Let $\varphi/K$ and $L_S/K$ be as in the theorem, and suppose there exists a place $v\in S$ such that $\mu_{\varphi,v}(\overline{L_v})<1$. Then the set $\PrePer_\varphi(L)$ of preperiodic points of $\varphi$ which are defined over the field $L$ of all totally $L_S/K$ numbers is finite.
\end{cor}

\begin{cor}\label{cor:full-measure}
 Let $\varphi/K$ and $L_S/K$ be as in the theorem, and $L$ the field of all totally $L_S/K$ numbers. If there exists a sequence of distinct algebraic numbers $\{\al_n\}\subset L$ such that $h_\varphi(\al_n)\ra 0$, then $\mu_{\varphi,v}(\overline{L_v})=1$ for every place $v\in S$.
\end{cor}

\begin{cor}\label{cor:totally-p-adic}
 Let $L_p/\bQ_p$ be a Galois extension with $L_p =\bR$ if $p=\infty$. Then 
\[
 \liminf_{\substack{\al\in\bP^1(\Qbar)\\ \al\text{ is totally }L_p}} h(\al)>0.
\]
\end{cor}

We note that for all but finitely many primes (i.e., those of good reduction), the local dynamics of $\varphi$ are trivial and resemble those of the standard height, so in particular we get the following:
\begin{cor}\label{cor:dyn-almost-all-p}
 Let $\varphi$ be a rational map defined over a number field $K$. Then for almost all finite places $v$ of $K$, we have
\[
 \liminf_{\substack{\al\in\bP^1(\Kbar)\\ \al\text{ is totally }L_v}} h_\varphi(\al)>0.
\]
 for any Galois extension $L_v/K_v$.
\end{cor}

Let $E/K$ be an elliptic curve defined over a number field $K$ given by
\[
 E : y^2 = x^3 + ax + b. 
\]
Baker and Petsche \cite[Theorems 17, 21]{BakerPetsche} proved that if $L/\bQ$ is a totally real field, or a field all of whose completions at the places lying over a prime $p$ have a bounded ramification and intertial degrees, then the infimum and limit infimum of $\hhat(P)$ for points $P\in E(L)$ are bounded away from zero by certain explicit constants.\footnote{In fact, in the case of bounded ramification, Baker proved this earlier in \cite[\S 5, Case 1]{BakerEllCurvAbelian}.} As an easy application of our theorem, we can obtain the following qualitative version of these results:
\begin{cor}\label{cor:ell-tot-p-adic}
 Let $K$ be a number field and $E/K$ be an elliptic curve. Let $S$ be a finite set of places of $K$ which are either archimedean or for which $E$ has good reduction. For each $v\in S$ we let $L_v/K_v$ be a (finite) Galois extension, with the assumption that $L_v=K_v=\bR$ if $v\mid \infty$. Then there exists a constant $c>0$ depending only on $E/K$ and the fields $L_v$ such that
 \[
  \liminf_{\substack{P\in E(\Kbar)\\ x(P)\text{ is totally }L_S/K}} \hhat(P)\geq c>0.
 \]
\end{cor}
\noindent We note that our hypothesis of good reduction is not always necessary, as the results of Baker and Petsche apply even in cases of semistable reduction. Naturally, this result as stated is strongest when we choose $S$ to be a single place, however, as happens in the bounds obtained by Baker and Petsche (akin to what we see in the theorem of Bombieri and Zannier), the constant should grow as the number of places $S$ increases. Determining the exact order of this growth would seem to be an interesting open question.

\subsection{Background on arithmetic dynamics}
Early work on the distribution of numbers by height focused on rational points, with results like the uniform distribution of the Farey fractions, culminating in quantitative generalization of the distribution of rational points at all places by Choi \cite{ChoiThesis,ChoiDistribution}. Shortly afterwards in a slightly different direction, Bilu \cite{Bilu} formulated an equidistribution theorem for algebraic points of small height at the archimedean place, proving that the Galois conjugates of a sequence of numbers with Weil height tending to zero must equidistribute along the unit circle in $\bC$ in the sense of weak convergence of measures. This result was later generalized to abelian varieties by Szpiro, Ullmo, and Zhang \cite{SUZ} and used in the proof of the Bogomolov conjecture by Zhang \cite{Zhang}. Later, these results were vastly generalized to dynamical heights and nonarchimedean places using the techniques of potential theory on Berkovich analytic spaces independently by Baker and Rumely \cite{BakerRumely}, Favre and Rivera-Letelier \cite{FRL, FRLcorrigendum}, and Chambert-Loir \cite{ChambertLoirEqui}. The proofs of our main theorems are rooted in this equidistribution result.

We briefly recall now the dynamical equidistribution theorem. Let $\varphi\in K(z)$ be a rational function of degree $d \geq 2$ for $K$ a number field. Then the \emph{dynamical (or canonical) height associated to $\varphi$}, first introduced by Call and Silverman \cite{CallSilverman}, is given by the Tate limit
\[
 h_{\varphi}(\al) = \lim_{n\ra\infty} \frac{1}{d^n} h(\varphi^n(\al))
\]
where as usual $\varphi^n = \varphi\circ \cdots \circ \varphi$ denotes the $n$-fold iteration. The dynamical height is characterized by the properties that
\begin{enumerate}
 \item The quantity $\abs{h_\varphi(\al) - h(\al)}$ is bounded by an absolute constant for all $\al\in \bP^1(\Qbar)$.
 \item $h_\varphi(\varphi(\al)) = d\cdot \,h_\varphi(\al)$ for all $\al\in \bP^1(\Qbar)$.
\end{enumerate}
\noindent The points $\al\in\bP^1(\Qbar)$ which satisfy $h_\varphi(\al)=0$ are precisely the preperiodic points of $\varphi$, that is, $\al$ which satisfy $\varphi^n(\al) = \varphi^m(\al)$ for some $n>m\geq 0$. In the case of $\varphi(z)=z^2$, for which $h_\varphi=h$, the preperiodic points are precisely $0$, $\infty$, and the roots of unity. We refer the reader to \cite{SilvDyn} for more background on dynamical heights.

Recent work, particularly regarding the equidistribution theorem which we will discuss in more detail below, has made it clear that the proper setting to study nonarchimedean dynamics appears to be the \emph{Berkovich analytic space} $\sP^1(\bC_v)$, where $\bC_v$ is the completion of an algebraic closure of $K_v$. For background on the Berkovich analytic space, we refer the reader to \cite{BerkovichBook,BakerRumelyBook}. While $\bP^1(\bC_v)$ is totally disconnected, not locally compact, and not spherically complete (there exist nested sequences of closed balls with empty intersection),  the Berkovich space $\sP^1(\bC_v)$ is a compact, Hausdorff, and path-connected, and it contains $\bP^1(\bC_v)$ as a dense subset. As a set, one can construct the affine Berkovich space $\sA^1(\bC_v)$ as the \emph{Berkovich spectrum} consisting of all bounded multiplicative seminorms on the normed ring $\bC_v[T]$, so each point in Berkovich space corresponds to such a seminorm. The set $\bC_v$ sits inside $\sA^1(\bC_v)$ as the evaluation seminorms $f\mapsto \abs{f(x)}_v$ for each $x\in \bC_v$. When $v\mid \infty$, this is the whole story: all bounded multiplicative seminorms come from evaluation maps at a point and $\bA^1_\Berk(\bC)=\bC$. For nonarchimedean $v$, however, there is a wealth of new points, most important amongst which in our study is the \emph{Gauss point} $\zeta_{0,1}$, which corresponds to the multiplicative seminorm
\[
 f\mapsto \sup_{\substack{z\in \bC_v\\ \abs{z}_v\leq 1}} \abs{f(z)}_v.
\]
One can check by Gauss's lemma that this does in fact define a multiplicative seminorm on $\bC_v[T]$. The Berkovich affine line is then endowed with the weakest topology for which the seminorms are continuous.

The equidistribution theorem states that any sequence of small points with respect to $\varphi$, that is, a sequence $\{\al_n\}$ of distinct numbers satisfying $h_\varphi(\al_n)\ra 0$, is equidistributed with respect to a certain \emph{canonical measure} $\mu_{\varphi,v}$ at each place $v$. More specifically, we recall the following result (see \cite{BakerRumely, FRL}):
\begin{thm*}[Dynamical equidistribution theorem]
 Let $\varphi$ be a rational map of degree at least 2 defined over the number field $K$. Let $\{\al_n\}$ be a sequence of distinct algebraic numbers satisfying $h_\varphi(\al_n)\ra 0$ and let $G_K=\Gal(\overline{K}/K)$. Let $v$ be a place of $K$, and for each $\al_n$ we define the probability measure on $\sP^1(\bC_v)$ supported equally on each of the $G_K$-conjugates of $\al_n$:
\[
 [\al_n] = \frac{1}{\# G_K\al_n} \sum_{z\in G_K \al_n} \delta_z
\]
where $\delta_z$ denotes the Dirac measure at $z$. If $\mu_{\varphi,v}$ denotes the canonical measure on $\sP^1(\bC_v)$ associated to $\varphi$ at $v$, then we have weak convergence of measures:
\[
 [\al_n] \mathop{\longrightarrow}^w \mu_{\varphi,v}\quad \text{as}\quad n\ra\infty,
\]
that is to say, for each continuous real or complex valued function $f$ on $\sP^1(\bC_v)$, we have
\[
 \frac{1}{\# G_K\al_n} \sum_{z\in G_K\al_n} f(z) \ra \int_{\sP^1(\bC_v)} f(z)\,d\mu_{\varphi,v}(z)\quad \text{as}\quad n\ra\infty.
\]
\end{thm*}
At a finite place $v$ of good reduction of $\varphi$, the measure $\mu_{\varphi,v} = \delta_{\zeta_{0,1}}$ is the Dirac measure supported at the Gauss point $\zeta_{0,1}$. We recall that any given rational $\varphi$ has good reduction at all but finitely many places of $K$. Bilu's theorem \cite{Bilu} is the special case of this theorem for the squaring map $\varphi(z)=z^2$, for which $h_\varphi = h$ and $\mu_{\varphi,\infty}$ is the normalized Haar measure of the unit circle in $\bC$. We note that already Bilu's theorem gives a rather vivid picture (Figure \ref{fig:unitcirc}) of why totally real numbers cannot have height tending to $0$, as they clearly cannot equidistribute around the unit circle in $\bC$. The proof of our main theorem is directly inspired by this simple picture.

\begin{figure}[ht!]
\centering
\includegraphics[width=0.5\textwidth]{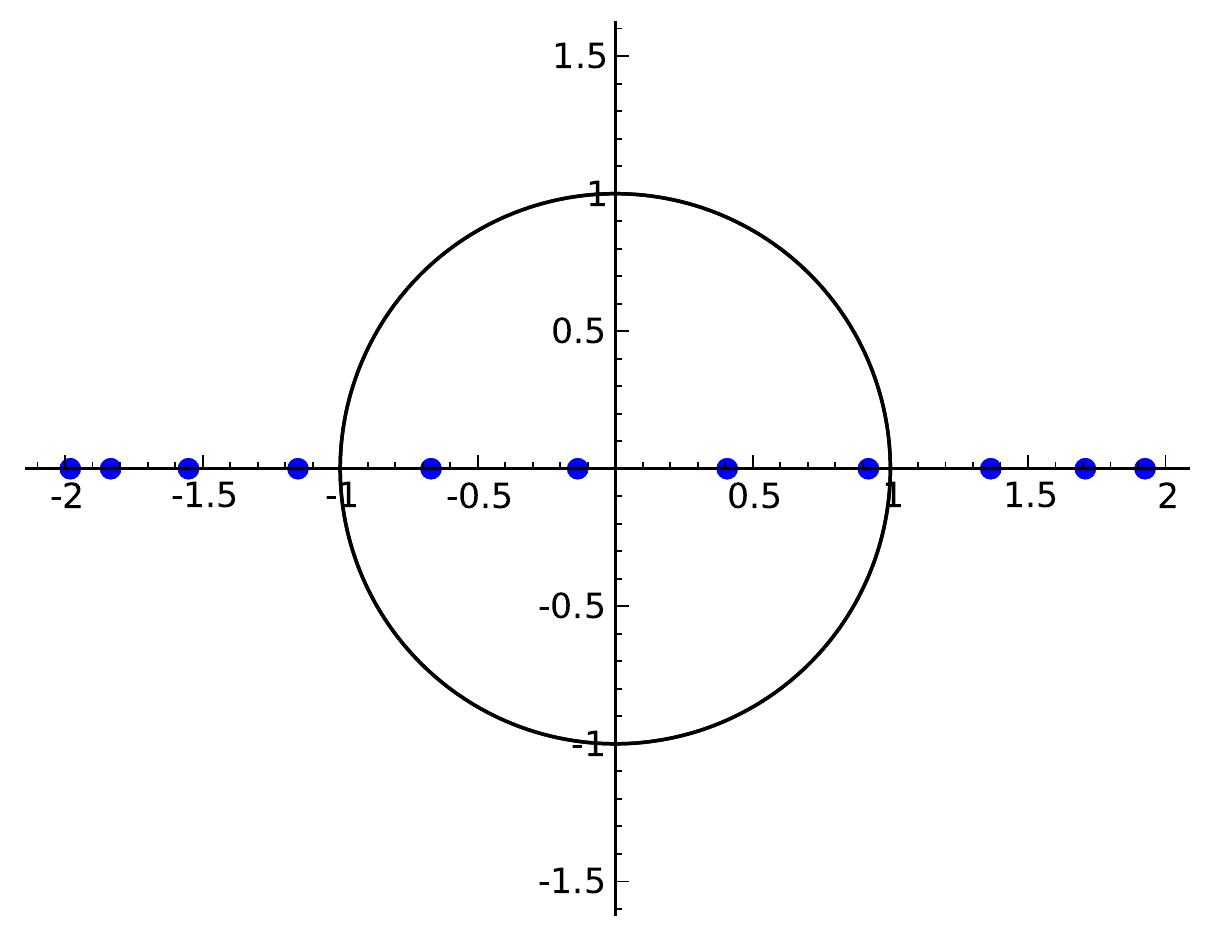} 
\label{fig:unitcirc}
\caption{The unit circle in $\bC$ and the Galois conjugates of the totally real number $\zeta_{23}+\zeta_{23}^{-1}$.}
\end{figure}

\section{Proofs}
\begin{proof}[Proof of Theorem \ref{thm:main}]
 Suppose for the sake of contradiction that there existed a sequence $\{\al_n\}\subset \bP^1(\Kbar)$ of distinct totally $L_S/K$ algebraic numbers such that $h_\varphi(\al_n)\ra 0$. Let $\mu_{\varphi,v}$ be the canonical measure and $G_K=\Gal(\Kbar/K)$. Fix an embedding $\Kbar\hookrightarrow \bC_v\subset \sP^1(\bC_v)$. Define the usual measures
\[
 [\al_n] = \frac{1}{\# G_K\al_n} \sum_{z\in G\al_n} \delta_z
\]
 where, for $z\in \bC_v$, $\delta_z$ denotes the Dirac measure at $z$. Then the equidistribution theorem tells us that we have weak convergence of measures
\[
 [\al_n] \mathop{\longrightarrow}^w \mu_{\varphi,v}.
\]
 
 Recall that we have natural continuous embeddings $\bP^1(\bC_v)\hookrightarrow \sP^1(\bC_v)$, and hence we can identify $L_v$ as a subset of $\sP^1(\bC_v)$. By assumption, $\mu_{\varphi,v}(\overline{L_v}^c) = 1 - \mu_{\varphi,v}(\overline{L_v}) > 0$, where $\overline{L_v}^c$ denotes the complement of $\overline{L_v}$ in $\sP^1(\bC_v)$. Since $\mu_{\varphi,v}$ is a regular measure, there exists a compact set $E\subset \overline{L_v}^c$ such that $\mu_{\varphi,v}(E)>0$ as well. As $E$ and $\overline{L_v}$ are disjoint closed subsets of $\sP^1(\bC_v)$, which is a compact Hausdorff space, we know by Urysohn's lemma that there exists a continuous function $f: \sP^1(\bC_v)\ra [0,1]$ which takes the value $1$ on $E$ and $0$ on $L_v$. But then
\[
 \int_{\sP^1(\bC_v)} f\,d[\al_n] = \frac{1}{\# G_K \al_n} \sum_{z\in G_K\al_n} f(z) = 0\quad\text{for all}\quad n,
\]
while
\[
 \int_{\sP^1(\bC_v)} f\, d\mu_{\varphi,v} \geq \mu_{\varphi,v}(E) > 0.
\]
Thus, for this function $f$,
\[
\lim_{n\ra\infty} \int_{\sP^1(\bC_v)} f\,d[\al_n] \neq \int_{\sP^1(\bC_v)} f\,d\mu_{\varphi,v} 
\]
but this is a contradiction to the weak convergence of measures, and completes the proof of the first part of the theorem.

Let us now prove the converse under the assumption that $\varphi$ is a polynomial. We will prove that if at all places $v\in S$, $\overline{L_v}\subset\sP^1(\bC_v)$ has full $\mu_{\varphi,v}$-measure, then there exists a sequence $\al_n\in\bP^1(\overline K)$ such that $h_\varphi(\al_n)\ra 0$.

Since $\varphi$ is a polynomial, it is a theorem (see for example \cite{BakerHsia,DeMarcoRumely,BakerRumely}) that the $v$-adic Berkovich Julia sets $J_v(\varphi)$ are compact in $\sA^1(\bC_v)$ and that the $v$-adic logarithmic capacity of each is given by
\[
 \gamma_{\infty,v}(J_v(\varphi)) = \abs{a_d}_v^{-1/(d-1)}
\]
where we have denoted by $a_d$ the leading coefficient of $\varphi$, that is, $\varphi(z)=a_d z^d + O(z^{d-1})$. Let \[\bJ = \prod_v J_v(\varphi)\] be the adelic Berkovich Julia set. Then the (normalized) adelic logarithmic capacity $\gamma_\infty(\bJ)=\prod_v \gamma_{\infty,v}(J_v)  = 1$ by the product formula. Our basic strategy will be to use the Fekete-Szeg\H o theorem with splitting conditions due to Rumely (\cite[Theorem 2.1]{RumelyFeketeII} and \cite[\S 6.7, 7.8]{BakerRumelyBook}; see also \cite{RumelyFeketeCurves}) to generate a sequence $\al_n\in L$ with $h_\varphi(\al_n)\ra 0$. Fix $\ep>0$ and let $S'$ be the union of $S$ with the infinite places of $K$ and the places of bad reduction of $\varphi$. Let $h_{\varphi,v}$ denote the usual $v$-adic local height associated to iteration of $\varphi$ and consider the adelic set 
\[
 \bU = \prod_{v\notin S'} \cD(0,1) \times \prod_{v\in S'} h_{\varphi,v}^{-1}([0,\ep)).
\]
where $\cD(0,1)$ is the closed $v$-adic Berkovich unit disc. Notice that for $v\notin S'$, $J_v(\varphi)=\{\zeta_{0,1}\}\subset\cD(0,1)$ is the usual Gauss point, and since $h_{\varphi,v}$ is nonnegative for each $v$, it follows that $\bU$ is an open adelic neighborhood of $\bJ$.

It now follows from \cite[Theorem 2.1]{RumelyFeketeII}, since at least one (in fact all) archimedean places $v$ we have $J_v(\varphi)$ is compact, that there are infinitely many totally $L_S$ algebraic numbers all whose Galois conjugates lie in the open adelic neighborhood $\bU$. Since any such number has height
$
 h_\varphi(\al) \leq \abs{S'} \cdot \ep,
$ by taking $\ep\ra 0$ we can generate an infinite sequence of mutually distinct totally $L_v$ algebraic numbers for all $v\in S$ with $h_\varphi(\al)\ra 0$, which is the desired conclusion.
\end{proof}
\begin{proof}[Proof Corollary \ref{cor:preperiodic-pts}]
 It is clear that the set of all totally $L_S/K$ numbers form a field $L$ which is normal over $K$. To see that the set of preperiodic points is finite, merely note that if an infinite number of such points existed, the limit infimum of the $\varphi$-canonical height would necessarily be $0$, contradicting Theorem \ref{thm:main}.
\end{proof}
\begin{proof}[Proof of Corollary \ref{cor:full-measure}]
 This is immediate as it is the contrapositive of the first part of the theorem.
\end{proof}
\begin{proof}[Proof of Corollary \ref{cor:totally-p-adic}]
 Let $\varphi(z)=z^2$ and $L_\infty=\bR$. Then the equilibrium measure of the (Berkovich) complex Julia set is $\mu_{\varphi,\infty} = \lambda$ is the normalized Haar measure of the unit circle $S^1$ in $\bC^\times\subset \bP^1_\Berk(\bC)=\bP^1(\bC)$. Clearly, the topological closure of $\bR$ in $\bP^1(\bC)$ is $\overline{\bR}=\bR\cup\{\infty\}$, but $\lambda(\overline{\bR}) = \lambda(\overline{\bR}\cap S^1) = \lambda(\{\pm 1\}) = 0$. Thus the theorem applies and we see that all totally $\bR$ numbers $\al\in\Qbar$, that is to say, all totally real numbers which are not zero or roots of unity (the preperiodic points of $\varphi$) must have Weil height bounded away from zero by an absolute constant.

 Now consider the case where $\varphi(z)=z^2$ and $L_p/\bQ_p$ is a Galois extension for a finite prime $p$. The equilibrium measure $\mu_{\varphi,p} = \delta_{\zeta_{0,1}}$ is the Dirac measure supported on the Gauss point $\zeta_{0,1}$. Now observe that $L_p\hookrightarrow \bP^1(L_p) \hookrightarrow \sP^1(\bC_v)$. But $\bP^1(L_p)$ is compact and the inclusion map is continuous, so $\bP^1(L_p)$ is compact (and in particular, closed) in $\sP^1(\bC_v)$. Therefore the Berkovich closure $\overline{L_p} = L_p \cup\{\infty\}$, where $\infty$ is the usual point $\infty \in \bP^1(\bC_v)$, but this clearly implies that $\zeta_{0,1}\not\in \overline{L_p}$. Therefore, $\mu_{\varphi,p}(\overline{L_v})= 0 <1$ and the theorem again applies.
\end{proof}

\begin{proof}[Proof of Corollary \ref{cor:dyn-almost-all-p}]
 We merely note that for almost all places of $K$, $\varphi$ has good reduction, which means that the Berkovich $v$-adic Julia set is the Gauss point. But then, as we argued in the proof of the above corollary, $\mu_{\varphi,v}(\overline{L_p})=0$, and so the theorem applies.
\end{proof}

Before we prove Corollary \ref{cor:ell-tot-p-adic}, let us recall some background on elliptic curves and Latt\`es maps. As before, let $E/K$ be an elliptic curve defined over a number field $K$ given by
\[
 E : y^2 = x^3 + ax + b. 
\]
Recall that the \emph{Latt\`es map} associated to the doubling map $[2]: E\ra E$, $[2](P)=2P$, is given by
\[
 \varphi(x) = \frac{x^4 - 2ax^2 - 8bx + a^2}{4x^3 +4ax +4b}.
\]
It is characterized by the fact that the following diagram commutes:
\[
\xymatrix{ 
 E(\Kbar)\ar[r]^{[2]}\ar[d]_x & E(\Kbar)\ar[d]^x \\
 \bP^1(\Kbar) \ar[r]_{\varphi} & \bP^1(\Kbar)
}
\]
where $x : E(\Kbar)\ra \bP^1(\Kbar)$ is the map onto the $x$-coordinate. If $\hhat : E(\Kbar)\ra \bR$ denotes the N\'eron-Tate height on $E$ and $h_\varphi$ is the dynamical height associated to the Latt\`es map, then
\[
 2\hhat(P) = h_\varphi(x(P)) \quad\text{for all}\quad P\in E(\Kbar).
\]
\begin{proof}[Proof of Corollary \ref{cor:ell-tot-p-adic}]
 Let us first consider the proof for $p=\infty$ and $L_p=\bR$. Recall that $\Etor(\bC)$ is dense in $E(\bC)$, so $x(\Etor(\bC))$ is dense in $\bP^1(\bC)$. It follows from this that the complex Julia set of $\varphi$ is $J_\infty = \bP^1(\bC)$. Thus, $\mu_{\varphi,\infty}(\bP^1(\bR))<1$, and our theorem implies that $P\in E(L)$ satisfy
 \[
 2\hhat(P)=h_\varphi(x(P))\geq c
 \]
 for some constant $c=c(L_p,E/K)>0$.
 
 Now suppose that $p\nmid \infty$. Let $\Delta = -16(4a^3+27b^2)$ be the discriminant of the elliptic curve. Recall that $E$ has good reduction at $v$ if and only if $\abs{\Delta}_v = 1$. We simply observe that $\Delta$ and the resultant of $\varphi$, $\Res(\varphi) = 256 \left(4 a^3+27 b^2\right)^2$, satisfy $\abs{\Delta}_v\neq 1\iff \abs{\Res(\varphi)}_v\neq 1$, and thus $E$ has good reduction if and only if $\varphi$ does. But then for some $v\mid p$, we know that the equilibrium measure $\mu_{\phi,v} = \delta_{\zeta_{0,1}}$ is the usual Dirac measure at the Gauss point and $\mu_{\phi,v}(L_p)=0$, so our theorem applies as before.
\end{proof}

\bibliographystyle{abbrv} 
\bibliography{bib-equi}        

\end{document}